\documentclass{article}[11pt]

\usepackage{amsmath}
\usepackage{amssymb}
\usepackage{amsthm}
\usepackage{enumerate}

\newtheorem{theorem}{Theorem}
\newtheorem{lemma}[theorem]{Lemma}
\newtheorem{prop}[theorem]{Proposition}
\newtheorem{cor}[theorem]{Corollary}

\newtheorem{definition}[theorem]{Definition}

\newtheorem{question}[theorem]{Question}

\newcommand{\dunion}{\dot{\cup}}

\newcommand{\hi}{H^{[1]}}

\allowdisplaybreaks
\date{\today}\begin{document}
\title{Chromatic Ramsey number of acyclic hypergraphs}

\author{Andr\'as Gy\'arf\'as\thanks{Advisor of the research conducted in the Research Opportunities Program of Budapest Semesters in Mathematics, Spring 2015.}\\
\small Alfr\'ed R\'enyi Institute of Mathematics\\[-0.8ex]
\small Hungarian Academy of Sciences\\[-0.8ex]
\small Budapest, P.O. Box 127 \\[-0.8ex]
\small Budapest, Hungary, H-1364 \\[-0.8ex]
\small \texttt{gyarfas@renyi.hu}\\[-0.8ex]\and
Alexander W.~N.~Riasanovsky\\[-0.8ex]
\small University of Pennsylvania\\[-0.8ex]
\small Department of Mathematics\\[-0.8ex]
\small Philadelphia, PA 19104, USA\\[-0.8ex]
\small \texttt{alexneal@math.upenn.edu}\and
Melissa U. Sherman-Bennett\\[-0.8ex]
\small Bard College at Simon's Rock\\[-0.8ex]
\small Department of Mathematics\\[-0.8ex]
\small Great Barrington, MA 01230, USA\\[-0.8ex]
\small \texttt{mshermanbennett12@simons-rock.edu}
}

\def\boxit#1{\medskip\vbox{\hrule \hbox{\vrule\kern15pt\vbox{\kern5pt
\vbox{\advance\hsize -30pt #1\par} \kern5pt}\kern15pt\vrule}\hrule} \medskip}

\maketitle

\eject
\begin{abstract}

Suppose that $T$ is an acyclic $r$-uniform hypergraph, with $r\ge 2$. We define the ($t$-color) chromatic Ramsey number $\chi(T,t)$ as the smallest $m$ with the following property: if the edges of any $m$-chromatic $r$-uniform hypergraph are colored with $t$ colors in any manner, there is a monochromatic copy of $T$. We observe that $\chi(T,t)$ is well defined and $$\left\lceil {R^r(T,t)-1\over r-1}\right \rceil +1 \le \chi(T,t)\le |E(T)|^t+1$$ where $R^r(T,t)$ is the $t$-color Ramsey number of $H$. We give linear upper bounds for $\chi(T,t)$ when T is a matching or star, proving that for $r\ge 2, k\ge 1, t\ge 1$, $\chi(M_k^r,t)\le (t-1)(k-1)+2k$ and $\chi(S_k^r,t)\le t(k-1)+2$ where $M_k^r$ and $S_k^r$ are, respectively, the $r$-uniform matching and star with $k$ edges.

The general bounds are improved for $3$-uniform hypergraphs. We prove that $\chi(M_k^3,2)=2k$, extending a special case of Alon-Frankl-Lov\'asz' theorem. We also prove that $\chi(S_2^3,t)\le t+1$, which is sharp for $t=2,3$. This is a corollary of a more general result. We define $H^{[1]}$ as the 1-intersection graph of $H$, whose vertices represent hyperedges and whose edges represent intersections of hyperedges in exactly one vertex. We prove that $\chi(H)\le \chi(H^{[1]})$ for any $3$-uniform hypergraph $H$ (assuming $\chi(H^{[1]})\ge 2$). The proof uses the list coloring version of Brooks' theorem.

\end{abstract}

\section{Introduction}

A hypergraph $H=(V,E)$ is a set $V$ of {\em vertices} together with a nonempty set $E$ of subsets of $V$, which are called {\em edges}.  In this paper, we will assume that for each $e\in E$, $|e|\geq 2$. If $|e|=r$ for each $e\in E$, then $H$ is {\em $r$-uniform}; a $2$-uniform $H$ is a graph.
 A hypergraph $H$ is {\em acyclic} if $H$ contains no cycles (including $2$-cycles which are two edges intersecting in at least two vertices). If $H$ is a connected acyclic hypergraph, we say that $H$ is a {\em tree}.  In particular, a {\em star} is a tree in which one vertex is common to every edge. A {\em matching} is a hypergraph consisting of pairwise disjoint edges, with every vertex belonging to some edge. We denote by $S_k^r$ and $M_k^r$  the $r$-uniform $k$-edge star and matching, respectively.

For a positive integer $k$, a function $c:V\to \{1,\dots,k\}$ is called a $k$-coloring of $H$. A coloring $c$ is {\em proper} if no edge of $H$ is monochromatic under $c$. The chromatic number of $H$, denoted $\chi(H)$, is the least $m\geq 1$ for which there exists a proper $m$-coloring of $H$ and in this case, we say that $H$ is $m$-chromatic. Given $H=(V,E)$, a partition $\{E_1,\dots,E_t\}$ of $E$ into $t$ parts is called a $t$-edge-coloring of $H$. For $r$-uniform hypergraphs $H_1,H_2,\dots,H_t$, the ($t$-color) Ramsey number $R^r(H_1,H_2,\dots,H_t)$ is the smallest integer $n$ for which the following is true: under any $t$-edge-coloring of the complete $r$-uniform hypergraph $K_n^r$, there is a monochromatic copy of $H_i$ in color $i$ for some $i\in \{1,2,\dots,t\}$.
When all $H_i=H$ we use the notation $R^r(H,t)$.

Bialostocki and the senior author of this paper extended two well-known results in Ramsey theory from complete host graph $K_n$ to arbitrary $n$-chromatic graphs \cite{BGY}. One extends a remark of Erd\H os and Rado stating that in any $2$-coloring of
the edges of a complete graph $K_n$ there is a monochromatic spanning tree. The other is the extension of the result of Cockayne and Lorimer \cite{CL} about the $t$-color Ramsey number of matchings. In \cite{G}, an acyclic graph $H$ is defined as {\em $t$-good} if every $t$-edge coloring  of any $R^2(H,t)$-chromatic graph contains a monochromatic copy of $H$. Matchings are $t$-good for every $t$ \cite{BGY} and  in \cite{G} it was proved that stars are $t$-good, as well as the path $P_4$ (except possibly for $t=3$). Additionally, $P_5,P_6,P_7$ are $2$-good. In fact, as remarked in \cite{BGY}, there is no known example of an acyclic $H$ that is not $t$-good.


In this paper, we explore a similar extension of Ramsey theory for hypergraphs, motivating the following definition.
\begin{definition}
Suppose that $T$ is an acyclic $r$-uniform hypergraph. Let $\chi(T,t)$ be the smallest $m$ with the following property: under any $t$-edge-coloring of any $m$-chromatic $r$-uniform hypergraph, there is a monochromatic copy of $T$.
\end{definition}

We call $\chi(T,t)$ the {\em chromatic Ramsey number of $T$}.
It follows from the existence of hypergraphs of large girth and chromatic number that the chromatic Ramsey number can be defined only for acyclic  hypergraphs.

\section{New results}

First we note that $\chi(T,t)$ is well-defined for any $r$-uniform tree $T$ and any $t\geq 1$, as an upper bound comes easily from the following result.
\vspace{12pt}

\noindent {\bf Lemma A.} (\cite{GYL},\cite{Loh})\label{k+1 bound}	If $H$ is $r$-uniform with $\chi(H)\geq k+1$, then $H$ contains a copy of any $r$-uniform tree on $k$ edges.
\begin{theorem} For any $r$-uniform tree $T$, $\chi(T,t)\le|E(T)|^t+1$.
\end{theorem}
\begin{proof}
Fix $t\geq 1$. Let $T$ be an $r$-uniform tree with $k$ edges and let $H=(V,E)$ be a hypergraph with $\chi(H)\geq k^t+1$.  Let $E=E_1\dunion\cdots\dunion E_t$ be a $t$-coloring of its edges.

 Then $\chi((V,E_1))\cdots\chi((V,E_t))\geq k^t+1$ and without loss of generality, $\chi((V,E_1))\geq k+1$.  Then by Lemma A, $(V,E_1)$ contains a copy of $T$.
\end{proof}
Since any $r$-uniform acyclic hypergraph $T$ may be found in some $r$-uniform tree $T^\prime$ and $\chi(T,t)\leq \chi(T^\prime,t)$, $\chi(T,t)$ is in fact well-defined for any $r$-uniform acyclic hypergraph and for any $t\geq 1$.  Observe the following natural lower bound of $\chi(T,t)$.  Let $L(T,t,r):=\left\lceil {R^r(T,t)-1\over r-1}\right \rceil +1$.
\begin{prop}\label{lowerb}
$L(T,t,r)\leq \chi(T,t)$
\end{prop}

\begin{proof}Let $N:=R^r(T,t)-1$. By the definition of the Ramsey number, there is a $t$-coloring of the edges of $K_N^r$ without a monochromatic $T$. Since $\chi(K_N^r)=\lceil {N\over r-1}\rceil$, the proposition follows.
\end{proof}

The notion of $t$-good graphs can be naturally extended to hypergraphs using Proposition \ref{lowerb}.
An acyclic $r$-uniform hypergraph $T$ is called {\em $t$-good} if every $t$-edge coloring of any $L(T,t,r)$-chromatic $r$-uniform hypergraph contains a monochromatic copy of $T$. In other words, $T$ is $t$-good if $L(T,t,r)=\chi(T,t)$. Note that for $r=2$, this gives the definition of good graphs. Although it is unlikely that all acyclic hypergraphs are $t$-good, we have no counterexamples.

For special families of $r$-uniform acyclic hypergraphs, we found linear upper bounds for $\chi(T,t)$, improving upon the general exponential upper bound above.  Surprisingly, most of the bounds attained do not depend on $r$.
\subsection{Matchings}
Indispensable in this section is the following well-known result of Alon, Frankl, and Lov\'asz (originally conjectured by Erd\H os).

\vspace{12pt}

\noindent {\bf Theorem B.} (\cite{AFL}) For $r\ge 2,k\ge 1,t\ge 1$, $$R^r(M_k^r,t)=(t-1)(k-1)+kr.$$
Note that special cases of Theorem B include $r=2$ \cite{CL}, $k=2$ \cite{LO}, $t=2$  \cite{AF}, \cite{GY}.

We obtain the following linear upper bound for matchings using Theorem B.
\begin{theorem}\label{genbound} For $r\ge 2,k\ge 1,t\ge 1$, $\chi(M_k^r,t)\le (t-1)(k-1)+2k$. Equality holds for $r=2$.
\end{theorem}
\begin{proof}
Let $H=(V,E)$ be an $r$-uniform hypergraph with  $\chi(H)=p$ where $p:=R^r(M_k^2,t)$.  By Theorem B, $p=(t-1)(k-1)+2k$. Consider any $t$-edge coloring $\{E_1,\dots,E_t\}$ of $H$ and any proper coloring $c$ of $H$ obtained by the greedy algorithm (under any ordering of its vertices). Clearly $c$ uses at least $p$ colors and for any $1\le i<j\le p$ there is an edge $e_{ij}$ in $H$ whose vertices are colored with color $i$ apart from a single vertex which is colored with $j$. Let $\{F_1,\dots,F_t\}$ be a $t$-edge-coloring of $K_p^2$ defined so that $F_s:=\{\{i,j\},\;1\leq i<j\leq p,\;e_{ij}\in E_s\}$ for each $s,\;1\leq s\leq t$.  From the definition of $p$, Theorem B (in fact the Cockayne-Lorimer Theorem suffices) implies that there is a monochromatic $M_k^2$ in $K_p$. Observe that
$$\{e_{ij}:\{i,j\}\in M_k^2\}$$
is a set of $k$ pairwise disjoint edges in $H$ in the same partition class of $\{E_1,\dots,E_t\}$. This completes the proof that $\chi(M_k^r,t)\le (t-1)(k-1)+2k$.  The lower bound $R^2(M_k^2,t)\leq\chi(M_k^2,t)$ implies equality in the $r=2$ case.
\end{proof}
Next we tighten this bound, provided $r\geq 3$ and $t=2$.
\begin{theorem}\label{main1}
For $r\ge 3$ and $k\geq 1$, $\chi(M_k^r,2)\leq 2k$.
\end{theorem}
\begin{proof}
We fix $r\geq 3$ and proceed by induction on $k$.  Suppose $k=1$ and let $H$ be some $r$-uniform hypergraph with $\chi(H)\geq 2$.  Then any $2$-edge-coloring of $H$ contains a single monochromatic edge since $H$ has at least one edge.  Now suppose the theorem is true for $k-1\geq 1$ and let $H=(V,E)$ be $r$-uniform with $\chi(H)\geq 2k$.  Without loss of generality, $H$ is connected.  Fix some $2$-edge-coloring $\{E_1,E_2\}$ of $H$, calling the edges of $E_1$ ``red'' and the edges of $E_2$ ``blue''.  If $E_1$ or $E_2$ is empty, then Theorem \ref{genbound} with $t=1$ implies the desired bound.  \\
\\So we may assume otherwise, and there exist edges $e,f\in E$ with $e$ red and $f$ blue.  Let $s:=|e\cap f|$ and $A:=e\cup f$.  If $H[A]$ is $2$-colorable, then $\chi(H-A)\geq\chi(H)-2=2(k-1)$ so by induction we find a monochromatic $M_{k-1}^r$ matching in $H-A$.  Without loss of generality, $M_{k-1}^r$ is red and $M_{k-1}^r+e$ is a red $M_k^r$ in $H$. 

If $s>1$, then $|A|=2r-s\leq 2r-2$ thus $H[A]$ is certainly $2$-colorable and the induction works.  If $s=1$  and $H[A]$ is not $2$-colorable then $H[A]$ is $K_{2r-1}^r$.  Writing $e=\{w,u_1,\dots,u_{r-1}\}$ and $f=\{w,v_1,\dots,v_{r-1}\}$,  the edge $g=\{w\}\cup\{u_1,u_3,\dots\}\cup\{v_2,v_4,\dots\}\in E(H)$.  Without loss of generality, $g$ is red and $|g\cap f|=1+\lfloor (r-1)/2\rfloor\geq2$ since $r\geq 3$.  So the previous case applies to the red edge $g$ and blue edge $f$. Finally, if $s=0$ and $H[A]$ is not $2$-colorable there must be $g\in H[A]$ that intersects both $e$ and $f$. Then either $e,g$ or $f,g$ is a pair of edges of different color that intersect, and a previous case can be applied again.
\end{proof}

\begin{cor} \label{aflcor}  $\chi(M_k^3,2)=2k$.
\end{cor}

\begin{proof}
The upper bound is given by Theorem \ref{main1}. The lower bound comes from Proposition \ref{lowerb} and Theorem B:
$$L(M_k^3,2,3)=\left\lceil {k-1+3k-1\over 2}\right\rceil+1=2k\le \chi(M_k^3,2).$$
 \end{proof}
\begin{cor} \label{2matching} For $r\ge 3$, $\chi(M_2^r,2)=4$.
\end{cor}

\begin{proof}As in Corollary \ref{aflcor}, the upper bound comes from Theorem \ref{main1} and the lower bound from
$$L(M_2^r,2,r)=\left\lceil{2r-1\over r-1}\right\rceil+1=4.$$
\end{proof}
It is worth noting that Corollary \ref{2matching} does not extend Theorem B to the chromatic Ramsey number setting for $r\ge4$. Indeed, for $r=4$, the lower bound $\lceil{2r\over r-1}\rceil+1$  of Proposition \ref{lowerb} is $4$ and the bound  $\lceil{1+2r\over r-1}\rceil$ derived from Theorem B is $3$.



\subsection{Stars}

\begin{theorem}\label{star}
 For $r\ge 2, k\ge 1, t\ge 1$, $\chi(S_k^r,t)\le t(k-1)+2$.
\end{theorem}

\begin{proof}
Fix $t,k\geq 1$ and let $p:=t(k-1)+2$. Suppose that $H$ is $r$-uniform with $\chi(H)\geq p$ and its edges are $t$-colored.  By Lemma A, $\chi(S_{p-1}^r,1)\leq p$, so we can find a copy of $S_{p-1}^r$ in $H$.  By the pigeonhole principle, $k$ of the edges of $S_{p-1}^r$ have the same color, and together they are a monochromatic copy of $S_k^r$. \end{proof}

How good is the estimate of Theorem \ref{star}? Notice first that for $t=1$ it is sharp.

\begin{prop} $\chi(S_k^r,1)=k+1$
\end{prop}

\begin{proof} Consider the complete hypergraph $K=K_{k(r-1)}^r$.  Clearly, $\chi(K)=k$ and $S_k^r$ is not a subgraph of $K$, as its vertex set is too large.
\end{proof}
If $t=2$, Theorem \ref{star} gives $\chi(S_k^r,2)\le 2k$. For $r=2$ and odd $k$, this is a sharp estimate. For $k=1$, this is trivial; for $k\geq 3$, the complete graph $K^2_{2k-1}$ can be partitioned into 2 $(k-1)$-regular subgraphs. However, for even $k\ge 2$, $\chi(S^2_k,2)=2k-1$.

An interesting problem arises when $T=S_2^r$ with $r\ge 3$, as Theorem \ref{star} gives the relatively low upper bound 4.  Can we decrease this bound? Namely:

\begin{question}\label{qubow} Is $\chi(S_2^r,2)=3$?
\end{question}

For $r=3$ the positive answer (Corollary \ref{speccase}) comes from a more general result, Theorem \ref{main2} below. We first need a definition.

\begin{definition}
Let $H=(V(H),E(H))$ be a hypergraph. The 1-intersection graph of $H$ is denoted $\hi$, where $V(\hi)=E(H)$ and $$E(\hi)=\{(e, f) : e,f\in E(H)\text{ and } |e \cap f| =1\}.$$
\end{definition}

It is well-known that if $\hi$ is trivial, i.e., no two edges of $H$ intersect in exactly one vertex, then $H$ is $2$-colorable (\cite{LO2}, Exercise 13.33).  
Note that the stronger statement  $\chi(H)\le \chi(\hi)+1 $ follows from applying the greedy coloring algorithm in any order of the vertices of  $H$.

\begin{question}\label{qu2int}
Let $r\ge3$. Is it true that $\chi(H)\le \chi(\hi)$ for any $r$-uniform hypergraph $H$, provided $\hi$ is nontrivial?
\end{question}
Our main result is the positive answer to Question \ref{qu2int} for the $3$-uniform case.

\begin{theorem} \label{main2} If $H$ is a $3$-uniform hypergraph with $\chi(\hi)\ge 2$ then $\chi(H)\le \chi(\hi)$.
\end{theorem}

\begin{cor}\label{speccase} For $t\ge 1$, $\chi(S_2^3,t)\le t+1$.
\end{cor}

The case $t=2$ of Corollary \ref{speccase} was the initial aim of the research in this paper and it was proved first by Zolt\'an F\"uredi \cite{FU}.  Our proof of Theorem \ref{main2} uses his observation (Lemma \ref{structure} below) and the list-coloring version of Brooks' theorem. Corollary \ref{speccase} is obviously sharp for $t=2$; it follows from Proposition \ref{lowerb} that it is also sharp for $t=3$, because $R^3(S_2^3,3)=6$ (\cite{AGYLM}).
 It would be interesting to see whether Corollary \ref{speccase} is true for any $S_2^r$ (in particular for $r=4,t=2$) as this is equivalent to the statement that $r$-uniform hypergraphs with bipartite 1-intersection graphs are 2-colorable.

\section{Proof of Theorem \ref{main2}}
In this section, we use the phrase ``triple system'' for a $3$-uniform hypergraph.  The word ``triple'' will take the place of ``edge'' so that ``edge'' may be reserved for graphs.  Our goal is to construct a proper $t$-coloring of $H$ from a proper $t$-coloring of $\hi$.  Note that a partition of $E(H)$ into classes $E_1, E_2,\dots, E_{t}$ such that for any $i,\;1\leq i\leq t,$ no two edges of $E_i$ $1$-intersect is precisely a proper $t$-coloring of $\hi$.  Let $B_k$ denote the triple system with $k$ edges intersecting pairwise in the vertices $\{v, w\}$, called the {\em base} of $B_k$. A $B$-component (also, $B_k$-component) is a triple system which is isomorphic to $B_k$ for some $k\ge 1$.  A $K$-component is either three or four distinct triples on four vertices. A triple system is connected if for every partition of its vertices into two nonempty parts, there is a triple intersecting both parts. Every triple system can be uniquely decomposed into pairwise disjoint connected parts, called components. Components with one vertex are called trivial components.

\begin{lemma} \label{structure} Let $C$ be a nontrivial component in a triple system without $1$-intersections. Then $C$ is either a $B$-component or a $K$-component.
\end{lemma}
\begin{proof}
If $C$ has at most four vertices then $1\leq |E(C)|\leq 4$ (where $E(C)$ is here considered as a set, not a multiset) and by inspection, $C$ is either $B_1,B_2$, or a $K$-component. Assume $C$ has at least five vertices and select the maximum $m$ such that $e_1,e_2,\dots e_m \in E(C)$ are distinct triples intersecting in a two-element set, say in $\{x,y\}$. Clearly, $m\ge 2$.  Then $A=\cup_{i=1}^m e_i$ must cover all vertices of $C$, as otherwise there is an uncovered vertex $z$ and  a triple $f$ containing $z$ and intersecting $A$, since $C$ is a component. However, from $m\ge 2$ and the intersection condition, $f\cap A=\{x,y\}$ follows, contradicting the choice of $m$. Thus $A=V(C)$ and from $|V(C)|\ge 5$ we have $m\ge 3$. It is obvious that any triple of $C$ different from the
$e_i$'s would intersect some $e_i$ in one vertex, violating the intersection condition. Thus $C$ is isomorphic to $B_m$, concluding the proof.
\end{proof}


A multigraph $G$ is called a {\em skeleton} of a triple system $H$ if every triple contains at least one edge of $G$. We may assume that $V(H)=V(G)$.  A {\em matching} in a multigraph is a set of pairwise disjoint edges. A {\em factorized complete graph } is a complete graph on $2m$ vertices whose edge set is partitioned into $2m-1$ matchings. The following lemma allows us to define a special skeleton of triple systems.

\begin{lemma}\label{goodcol}
Suppose that $H$ is a triple system with $\chi(H^{[1]})=t\ge 2$ and let $H_1,H_2,\dots,H_t$ be a partition of $H$ into triple systems without $1$-intersections.  There exists a skeleton $G$ of $H$ with the following properties.
\begin{enumerate}

\item $E(G)=\cup_{i=1}^t M_i$ where each $M_i$ is a matching and a skeleton of $H_i$.

\item For $1\le i \le t$, edges of $M_i$ are the bases of all $B$-components of $H_i$ and two disjoint vertex pairs from all $K$-components of $H_i$.

\item If $K^*=K_{t+1}\subset G$ then $K^*$ is a connected component of $G$ factorized by the $M_i$'s and there is $e\in M_1\cap E(K^*)$ such that $e$ is from a $B$-component of $H_1$.

\end{enumerate}
\end{lemma}
\begin{proof} From Lemma \ref{structure} we can define $M_i$ by selecting the base edges from every $B$-component of $H_i$ and selecting two disjoint pairs from every $K$-component of $H_i$. The resulting multigraph is clearly a skeleton of $H$ and satisfies properties 1 and 2. We will select the disjoint pairs from the $K$-components so that property 3 also holds. Notice that $K^*=K_{t+1}\subset G$ must form a connected component in $G$ because it is a $t$-regular subgraph of a graph of maximum degree $t$. Also, $K_{t+1}$ is factorized by the $M_i$'s because the union of $t$ matchings can cover at most ${t(t+1)\over 2}={t+1\choose 2}$ edges of $K_{t+1}$, therefore every edge of $K_{t+1}$ must be covered exactly once by the $M_i$'s.
Thus we have to ensure only that there is $e\in M_1\cap E(K^*)$ with $e$ from a $B$-component of $H_1$. For convenience, we say that a $K^*=K_{t+1}$ is a {\em bad component} if such $e$ does not exist.
 
Select a skeleton $S$ as described in the previous paragraph such that $p$, the number of bad components, is as small as possible.  Suppose that $(x,y)\in M_1$ is in a bad component $U$. In other words, $(x,y)$ is in a $K$-component of $H_1$, where $V(K)=\{x,y,u,v\}$ and $(u,v)\in M_1$.  Now we replace these two pairs by the pairs $(x,u),(y,v)$ to form a new $M_1$. After this switch, $U$ is no longer a bad component.  In fact, either $U$ becomes a new component on the same vertex set (if $(u,v)$ was in $U$) or $U$ melds with another component into a new component. In both cases, no new bad components are created and in the new skeleton there are fewer than $p$ bad components. This contradiction shows that $p=0$ and proves the lemma.
\end{proof}

\begin{proof}[Proof of Theorem \ref{main2}]

Let $H$ be a triple system with $t:=\chi(\hi)\ge 2$ and partition $H$ into $H_1,\dots,H_t$ so that each $H_i$ is without $1$-intersections. Let $G$ be a skeleton of $H$ with the properties ensured by Lemma \ref{goodcol}.

Let $G'$ be a connected component  of $G$. By Brooks' Theorem, if $G'$ is not the complete graph $K_{t+1}$ or an odd cycle (if $t=2$), $\chi(G')\leq \Delta(G')\leq t$.

Suppose first that $t$ is even. Now $G'\ne K_{t+1}$ because that would contradict property 3 in Lemma \ref{goodcol}: $K_{t+1}$ cannot be factorized into matchings.  Also, for $t=2$, $G'$ cannot be an odd cycle since odd cycles are not the union of two matchings. Thus every connected component of $G$ is at most $t$-chromatic, therefore $\chi(G)\le t$. Since $G$ is a skeleton of $H$, this implies $\chi(H)\le t$, concluding the proof for the case when $t$ is even.

Suppose that $t$ is odd, $t\ge 3$. In this case the previous argument does not work when some connected component $G'=K_{t+1}\subset G$.  However, from Lemma \ref{goodcol}, every $K_{t+1}$-component $C_i$ of $G$ has an edge $(x_i,y_i)\in M_1$ that is the base of a $B$-component in $H_1$. Define the vertex coloring $c$ on $X=\cup_{i=1}^m V(C_i)$ by $c(x_i)=c(y_i)=1$ and by coloring all the other vertices of all $C_i$'s with $2,\dots,t$.

Let $F$ be the subgraph of $G$ spanned by $V(G)\setminus X$ and define

\[
Z:=\{z\in V(F): \{x_i,y_i,z\}\in E(H_1) \text{ for some } 1\le i\le m\}.
\]

Since for every $z\in Z$ there is a triple $T=(x_i,y_i,z)\in H_1$ in a $B$-component of $H_1$ (with base $(x_i,y_i)$),  $(z,u)\in M_1$ is impossible for any $u\in V(G)$, since otherwise $T$ and the triple of $H_1$ containing $(z,u)$ would $1$-intersect in $z$. Thus $d_G(v)\le t-1$ for $z\in Z$. Also, $d_G(v)\le t$ for all $v\in V(F)\setminus Z$.

We claim that with lists $L(z):=\{2,\dots,t\}$ for $z\in Z$ and $L(v):=\{1,\dots,t\}$ for $v\in V(F)\setminus Z$, $F$ is $L$-choosable.  We use the reduction argument present in many coloring proofs (see, for example, the very recent survey paper \cite{CR}).

Suppose $F$ is not $L$-choosable and let $F^\prime$ be a minimal induced subgraph of $F$ which fails to be $L$-choosable.  We may assume that any $z\in V(F^\prime)\cap Z$ has $d_{F^\prime}(z)=t-1$ (otherwise we may $L$-choose $F^\prime-z$, add $z$ back and properly color it).  Likewise we may assume $d_{F^\prime}(v)=t$ for all $v\in V(F^\prime)\setminus Z$.  By the degree-choosability version of Brooks' theorem (see \cite{KSW}, Lemma 1 or  \cite{CR}, Theorem 11), $F^\prime$ is a Gallai tree: a graph whose blocks are complete graphs or odd cycles.

Let $A$ be a block of $F^\prime$.  Then $A\ne K_{t+1}$ because all $K_{t+1}$-components of $G$ are in $X$. Since all vertex degrees in $F'$ are $t$ or $t-1$, $A$ is either an odd cycle (if $t=3$) or $A$ is a $K_t$. $A$ must contain an edge $e\in M_1$. Otherwise $M_2,\dots,M_t$ would cover the edges of $A$, a contradiction in either case. If A is an endblock then by the degree requirements, either $$V(A)\cap (V(F)\setminus Z)=\{w\}$$ where $w$ is the unique cut point of $A$ or $V(A)\subset Z$. In both cases an endpoint of $e$ must be in $Z$. Then there exists some triple $\{x_i,y_i,z\}\in H_1$ which $1$-intersects with the triple of $H_1$ containing $e$, a contradiction, proving that $F$ is $L$-choosable.

Let $c^\prime:V(F)\to\{1,\dots,t\}$ be an $L$-coloring of $F$.  We extend $c$ from $X$ to $V(H)$ by setting $c(v):=c^\prime(v)$ for all $v\in V(F)$.  Observe that $c$ properly colors all edges of $G$ except for the edges of the form $(x_i,y_i)$ which are monochromatic in color $1$.
Since $G$ is a skeleton, every triple of $H$ is properly colored except possibly the triples in the from $(x_i,y_i,x)$.

We claim that $c(x)\ne 1$.
Suppose to the contrary that $c(x)=1$. If $x\in X$ then  $x\in \{x_j,y_j\}$ for some $j\ne i$, but this is impossible because the bases $(x_i,y_i), (x_j,y_j)$ are from different $B$-components of $H_1$. If $x\notin X$ then $x\in Z$ from the definition of $Z$. However, $1\notin L(x)$ for $x\in Z$ and this proves the claim.

Therefore $c$ is a proper $t$-coloring of $H$ and this completes the proof.
\end{proof}

\eject

\end{document}